\documentclass{birkmult}
\usepackage{hyperref, url}
\usepackage{soul,color}
 \newtheorem{theorem}{Theorem}[section]
 
 \newtheorem{lemma}[theorem]{Lemma}
 \newtheorem{proposition}[theorem]{Proposition}
 
 \theoremstyle{definition}
 
 \theoremstyle{remark}

 \numberwithin{equation}{section}

\newcommand{\R}{\mathbb{R}}

\newcommand{\Z}{\mathbb{Z}}

\newcommand{\cB}{\mathcal{B}}

\newcommand{\cK}{\mathcal{K}}

\newcommand{\cM}{\mathcal{M}}

\newcommand{\Op}{\operatorname{Op}}
\begin{document}
%
%
%
%
%
%
%
%
%
\title[PDO on Variable Lebesgue Spaces]
{Pseudodifferential Operators\\
 on Variable Lebesgue Spaces}
\author[A.~Yu.~Karlovich]{Alexei Yu. Karlovich}
\address{
Departamento de Matem\'atica\\
Faculdade de Ci\^encias e Tecnologia\\
Universidade Nova de Lisboa\\
Quinta da Torre\\
2829--516 Caparica\\
Portugal}
\email{oyk@fct.unl.pt}

\author[I.~M.~Spitkovsky]{Ilya M. Spitkovsky}
\address{
Department of Mathematics\\
College of William \& Mary\\
Williamsburg, VA, 23187-8795\\
U.S.A.} \email{ilya@math.wm.edu}



\dedicatory{To Professor Vladimir Rabinovich on the occasion of
his 70th birthday}
\begin{abstract}
Let $\mathcal{M}(\mathbb{R}^n)$ be the class of bounded away from
one and infinity functions $p:\mathbb{R}^n\to[1,\infty]$ such that
the Hardy-Littlewood maximal operator is bounded on the variable
Lebesgue space $L^{p(\cdot)}(\mathbb{R}^n)$. We show that if $a$
belongs to the H\"ormander class $S_{\rho,\delta}^{n(\rho-1)}$
with $0<\rho\le 1$, $0\le\delta<1$, then the pseudodifferential
operator $\Op(a)$ is bounded on the variable Lebesgue space
$L^{p(\cdot)}(\R^n)$ provided that $p\in\cM(\R^n)$. Let
$\mathcal{M}^*(\mathbb{R}^n)$ be the class of variable exponents
$p\in\mathcal{M}(\mathbb{R}^n)$ represented as
$1/p(x)=\theta/p_0+(1-\theta)/p_1(x)$ where $p_0\in(1,\infty)$,
$\theta\in(0,1)$, and $p_1\in\mathcal{M}(\mathbb{R}^n)$. We prove
that if $a\in S_{1,0}^0$ slowly oscillates at infinity in the
first variable, then the condition
\[
\lim_{R\to\infty}\inf_{|x|+|\xi|\ge R}|a(x,\xi)|>0
\]
is sufficient for the Fredholmness of $\Op(a)$ on $L^{p(\cdot)}(\R^n)$
whenever $p\in\cM^*(\R^n)$. Both theorems generalize pioneering
results by Rabinovich and Samko \cite{RS08} obtained for globally
log-H\"older continuous exponents $p$, constituting a proper subset
of $\mathcal{M}^*(\mathbb{R}^n)$.
\end{abstract}
\maketitle
\section{Introduction}
We denote the usual operators of first order partial
differentiation on $\R^n$ by
$\partial_{x_j}:=\partial/\partial_{x_j}$. For every multi-index
$\alpha=(\alpha_1,\dots,\alpha_n)$ with non-negative integers
$\alpha_j$, we write
$\partial^\alpha:=\partial_{x_1}^{\alpha_1}\dots\partial_{x_n}^{\alpha_n}$.
Further, $|\alpha|:=\alpha_1+\dots+\alpha_n$, and for each vector
$\xi=(\xi_1,\dots,\xi_n)\in\R^n$, define
$\xi^\alpha:=\xi_1^{\alpha_1}\dots\xi_n^{\alpha_n}$ and
$\langle\xi\rangle:=(1+|\xi|_2^2)^{1/2}$ where $|\xi|_2$ stands
for the Euclidean norm of $\xi$.

Let $C_0^\infty(\R^n)$ denote the set of all infinitely
differentiable functions with compact support. Recall that, given
$u\in C_0^\infty(\R^n)$, a pseudodifferential operator $\Op(a)$ is
formally defined by the formula
\[
(\Op(a)u)(x):=\frac{1}{(2\pi)^n}\int_{\R^n}d\xi\int_{\R^n}a(x,\xi)u(y)e^{i\langle x-y,\xi\rangle}dy,
\]
where the symbol $a$ is assumed to be smooth in both the spatial variable $x$ and
the frequency variable $\xi$, and satisfies certain growth conditions (see e.g.
\cite[Chap.~VI]{S93}). An example
of symbols one might consider is the class $S_{\rho,\delta}^m$, introduced by
H\"ormander \cite{H67}, consisting of $a\in C^\infty(\R^n\times\R^n)$
with
\[
|\partial_\xi^\alpha\partial_x^\beta a(x,\xi)|\le C_{\alpha,\beta}\langle\xi\rangle^{m-\rho|\alpha|+\delta|\beta|}
\quad (x,\xi\in\R^n),
\]
where $m\in\R$ and $0\le\delta,\rho\le 1$ and the positive constants $C_{\alpha,\beta}$
depend only on $\alpha$ and $\beta$.

The study of pseudodifferential operators $\Op(a)$ with symbols in $S_{1,0}^0$
on so-called variable Lebesgue spaces was started by Rabinovich and Samko \cite{RS08,RS11}.

Let $p:\R^n\to[1,\infty]$ be a measurable a.e. finite function. By
$L^{p(\cdot)}(\R^n)$ we denote the set of all complex-valued
functions $f$ on $\R^n$ such that
\[
I_{p(\cdot)}(f/\lambda):=\int_{\R^n} |f(x)/\lambda|^{p(x)} dx <\infty
\]
for some $\lambda>0$. This set becomes a Banach space when
equipped with the norm
\[
\|f\|_{p(\cdot)}:=\inf\big\{\lambda>0: I_{p(\cdot)}(f/\lambda)\le 1\big\}.
\]
It is easy to see that if $p$ is constant, then $L^{p(\cdot)}(\R^n)$ is nothing but
the standard Lebesgue space $L^p(\R^n)$. The space $L^{p(\cdot)}(\R^n)$
is referred to as a \textit{variable Lebesgue space}.
\begin{lemma}{\em (see e.g. \cite[Theorem~2.11]{KR91} or \cite[Theorem~3.4.12]{DHHR11})}
\label{le:density} If $p:\R^n\to[1,\infty]$ is an essentially
bounded measurable function, then $C_0^\infty(\R^n)$ is dense in
$L^{p(\cdot)}(\R^n)$.
\end{lemma}
We will always suppose that
\begin{equation}\label{eq:exponents}
1<p_-:=\operatornamewithlimits{ess\,inf}_{x\in\R^n}p(x),
\quad
\operatornamewithlimits{ess\,sup}_{x\in\R^n}p(x)=:p_+<\infty.
\end{equation}
Under these conditions, the space $L^{p(\cdot)}(\R^n)$ is
separable and reflexive, and its dual space is isomorphic to
$L^{p'(\cdot)}(\R^n)$, where
\[
1/p(x)+1/p'(x)=1 \quad(x\in\R^n)
\]
(see e.g. \cite{KR91} or \cite[Chap.~3]{DHHR11}). 

Given $f\in L_{\rm loc}^1(\R^n)$, the Hardy-Littlewood maximal operator is
defined by
\[
Mf(x):=\sup_{Q\ni x}\frac{1}{|Q|}\int_Q|f(y)|dy
\]
where the supremum is taken over all cubes $Q\subset\R^n$
containing $x$ (here, and throughout, cubes will be assumed to
have their sides parallel to the coordinate axes). By $\cM(\R^n)$
denote the set of all measurable functions $p:\R^n\to[1,\infty]$
such that \eqref{eq:exponents} holds and the Hardy-Littlewood
maximal operator is bounded on $L^{p(\cdot)}(\R^n)$.

Assume that \eqref{eq:exponents} is fulfilled. Diening \cite{D04}
proved that if $p$ satisfies
\begin{equation}\label{eq:log-Hoelder}
|p(x)-p(y)|\le\frac{c}{\log(e+1/|x-y|)}\quad (x,y\in\R^n)
\end{equation}
and $p$ is constant outside some ball, then $p\in\cM(\R^n)$. Further, the behavior
of $p$ at infinity was relaxed by Cruz-Uribe, Fiorenza, and Neugebauer \cite{CFN03,CFN04},
where it was shown that if $p$ satisfies \eqref{eq:log-Hoelder} and
there exists a $p_\infty>1$ such that
\begin{equation}\label{eq:Hoelder-infinity}
|p(x)-p_\infty|\le\frac{c}{\log(e+|x|)}\quad (x\in\R^n),
\end{equation}
then $p\in\cM(\R^n)$. Following \cite[Section~4.1]{DHHR11}, we will
say that if conditions
\eqref{eq:log-Hoelder}--\eqref{eq:Hoelder-infinity} are fulfilled, then
$p$ is {\em globally log-H\"older continuous}.

Conditions \eqref{eq:log-Hoelder} and \eqref{eq:Hoelder-infinity} are optimal
for the boundedness of $M$ in the  pointwise sense; the corresponding
examples are contained in \cite{PR01} and \cite{CFN03}.
However, neither \eqref{eq:log-Hoelder} nor \eqref{eq:Hoelder-infinity}
is necessary for $p\in\cM(\R^n)$. Nekvinda \cite{N04} proved that if $p$
satisfies \eqref{eq:exponents}--\eqref{eq:log-Hoelder} and
\begin{equation}\label{eq:Nekvinda}
\int_{\R^n}|p(x)-p_\infty|c^{1/|p(x)-p_\infty|}\,dx<\infty
\end{equation}
for some $p_\infty>1$ and $c>0$, then $p\in\cM(\R^n)$. One can show that
\eqref{eq:Hoelder-infinity} implies \eqref{eq:Nekvinda}, but the converse,
in general, is not true. The corresponding example is constructed in \cite{CCF07}.
Nekvinda further relaxed condition \eqref{eq:Nekvinda} in \cite{N08}.
Lerner \cite{L05} (see also \cite[Example~5.1.8]{DHHR11})
showed that there exist discontinuous at zero or/and at infinity exponents,
which nevertheless belong to $\cM(\R^n)$. We refer to
the recent monograph \cite{DHHR11} for further discussions concerning the class $\cM(\R^n)$.

Our first main result is the following theorem on the boundedness of pseudodifferential
operators on variable Lebesgue spaces.
\begin{theorem}\label{th:boundedness} Let $0<\rho\le 1$, $0\le\delta<1$, and
$a\in S_{\rho,\delta}^{n(\rho-1)}$. If $p\in\cM(\R^n)$, then
$\Op(a)$ extends to a bounded operator on the variable Lebesgue
space $L^{p(\cdot)}(\R^n)$.
\end{theorem}
The respective result for $a\in S_{1,0}^0$ and $p$ satisfying
\eqref{eq:exponents}--\eqref{eq:Hoelder-infinity} was proved by
Rabinovich and Samko \cite[Theorem~5.1]{RS08}.

Following \cite[Definition~4.5]{RS08}, a symbol $a\in S_{1,0}^m$ is
said to be {\em slowly oscillating at infinity} in the first variable if
\[
|\partial_\xi^\alpha\partial_x^\beta a(x,\xi)| \le
C_{\alpha\beta}(x)\langle\xi\rangle^{m-|\alpha|},
\]
where
\begin{equation}\label{eq:SO}
\lim_{x\to\infty}C_{\alpha\beta}(x)=0
\end{equation}
for all multi-indices $\alpha$ and $\beta\ne 0$. We denote by
$SO^m$ the class of all symbols slowly oscillating at infinity.
Finally, we denote by $SO_0^m$ the set of all symbols $a\in SO^m$,
for which \eqref{eq:SO} holds for all multi-indices $\alpha$ and
$\beta$. The classes $SO^m$ and $SO_0^m$ were introduced by
Grushin \cite{G70}.

We denote by $\cM^*(\R^n)$ the set of all variable exponents
$p\in\cM(\R^n)$ for which there exist constants
$p_0\in(1,\infty)$, $\theta\in(0,1)$, and a variable exponent
$p_1\in\cM(\R^n)$ such that
\[
\frac{1}{p(x)}=\frac{\theta}{p_0}+\frac{1-\theta}{p_1(x)}
\]
for almost all $x\in\R^n$.  Rabinovich and Samko observed in the
proof of \cite[Theorem~6.1]{RS08} that if $p$ satisfies
\eqref{eq:exponents}--\eqref{eq:Hoelder-infinity}, then
$p\in\cM^*(\R^n)$. It turns out that the class $\cM^*(\R^n)$
contains many interesting exponents which are not globally
log-H\"older continuous (see \cite{KS11-example}). In particular,
there exists  $\varepsilon>0$ such that for every $\alpha,\beta$
satisfying $0<\beta<\alpha\le\varepsilon$ the function
\[
p(x)=2+\alpha+\beta\sin\big(\log(\log|x|)\chi_{\{x\in\mathbb{R}^n:|x|\ge e\}}(x)\big)
\quad (x\in\mathbb{R}^n)
\]
belongs to $\mathcal{M}^*(\mathbb{R}^n)$.

As usual, we denote by $I$ the identity operator on a Banach
space. Recall that a bounded linear operator $A$ on a Banach space
is said to be Fredholm if there is an (also bounded linear)
operator $B$ such that the operators $AB-I$ and $BA-I$ are
compact. In that case the operator $B$ is called a {\em
regularizer} for the operator $A$.

Our second main result is the following sufficient condition for
the Fredholmness of pseudodifferential operators on variable
Lebesgue spaces.
\begin{theorem}\label{th:sufficiency}
Suppose $p\in\cM^*(\R^n)$ and $a\in SO^0$. If
\begin{equation}\label{eq:sufficiency-1}
\lim_{R\to\infty}\inf_{|x|+|\xi|\ge R}|a(x,\xi)|>0,
\end{equation}
then the operator $\Op(a)$ is Fredholm on the variable Lebesgue
space $L^{p(\cdot)}(\R^n)$.
\end{theorem}
As it was the case with Theorem~\ref{th:boundedness}, for $p$
satisfying \eqref{eq:exponents}--\eqref{eq:Hoelder-infinity} this
result was established by Rabinovich and Samko
\cite[Theorem~6.1]{RS08}. Notice that for such $p$ condition
\eqref{eq:sufficiency-1} is also necessary for the Fredholmness
(see \cite[Theorems~6.2 and~6.5]{RS08}). Whether or not the
necessity holds in the setting of Theorem~\ref{th:sufficiency},
remains an open question.

The paper is organized as follows. In Section~\ref{sec:sharp}, the
Diening-R\r{u}\v zi\v cka generalization (see \cite{DR03}) of the
Fefferman-Stein sharp maximal theorem to the variable exponent
setting is stated. Further, Diening's results \cite{D05} on the duality
and left-openness of the class $\cM(\R^n)$  are formulated. In
Section~\ref{sec:pointwise} we discuss a pointwise estimate relating
the Fefferman-Stein sharp maximal operator of $\Op(a)u$ and
$M_qu:=M(|u|^q)^{1/q}$ for $q\in(1,\infty)$ and $u\in C_
0^\infty(\R^n)$. Such an estimate for the range of parameters $\rho$,
$\delta$, and $m=n(\rho-1)$ as in Theorem~\ref{th:boundedness} was
recently obtained by Michalowski, Rule, and Staubach \cite{MRS11}.
Combining this key pointwise estimate with the sharp maximal
theorem and taking into account that $M_q$ is bounded on
$L^{p(\cdot)}(\R^n)$ for some $q\in(1,\infty)$ whenever
$p\in\cM(\R^n)$, we give the proof of Theorem~\ref{th:boundedness}
in Section~\ref{sec:proof}.

Section~\ref{sec:Fredholmness} is devoted to the proof of the
sufficient condition for the Fredholmness of a pseudodifferentail
operator with slowly oscillating symbol. In
Section~\ref{sec:interpolation}, we state analogues of the
Riesz-Thorin and Krasnoselskii interpolation theorems for variable
Lebesgue spaces. Section~\ref{sec:calculus} contains the
composition formula for pseudodifferential operators with slowly
oscillating symbols and the compactness result for
pseudodifferential operators with symbols in $SO_0^{-1}$. Both
results are essentially due to Grushin \cite{G70}.
Section~\ref{sec:sufficiency} contains the proof of
Theorem~\ref{th:sufficiency}. Its outline is as follows. From
\eqref{eq:sufficiency-1} it follows that there exist symbols
$b_R\in SO^0$ and $\varphi_R+c\in SO_0^{-1}$ such that
$I-\Op(a)\Op(b_R)=\Op(\varphi_R+c)$. Since $\varphi_R+c\in
SO_0^{-1}$, the operator $\Op(\varphi_R+c)$ is compact on all
standard Lebesgue spaces. Its compactness on the variable Lebesgue
space $L^{p(\cdot)}(\R^n)$ is proved by interpolation, since it is
bounded on the variable Lebesgue space $L^{p_1(\cdot)}(\R^n)$,
where $p_1$ is the variable exponent from the definition of the
class $\cM^*(\R^n)$. Actually, the class $\cM^*(\R^n)$ is
introduced exactly for the purpose to perform this step. Therefore
$\Op(b_R)$ is a right regularizer for $\Op(a)$ on
$L^{p(\cdot)}(\R^n)$. In the same fashion it can be shown that
$\Op(b_R)$ is a left regularizer for $\Op(a)$. Thus $\Op(a)$ is
Fredholm.
\section{Boundedness of the operator \boldmath{${\rm Op}(a)$}}
\subsection{Lattice property of variable Lebesgue spaces}
We start with the following simple but important property of variable Lebesgue
spaces. Usually it is called the lattice property or the ideal property.
\begin{lemma}{\em (see e.g. \cite[Theorem~2.3.17]{DHHR11})}
\label{le:lattice} Let $p:\R^n\to[1,\infty]$ be a measurable a.e.
finite function. If $g\in L^{p(\cdot)}(\R^n)$, $f$ is a measurable
function, and $|f(x)|\le|g(x)|$ for a.e. $x\in\R^n$, then $f\in
L^{p(\cdot)}(\R^n)$ and $\|f\|_{p(\cdot)}\le \|g\|_{p(\cdot)}$.
\end{lemma}
\subsection{The Fefferman-Stein sharp maximal function}
\label{sec:sharp}
Let $f\in L^1_{\rm loc}(\mathbb{R}^n)$.  For a cube
$Q\subset\mathbb{R}^n$, put
\[
f_Q:=\frac{1}{|Q|}\int_Q f(x)dx.
\]
The Fefferman-Stein sharp maximal function is defined by
\[
M^\#f(x):=\sup_{Q\ni x}\frac{1}{|Q|}\int_Q|f(x)-f_Q|dx,
\]
where the supremum is taken over all cubes $Q$ containing $x$.

It is obvious that $M^\# f$ is pointwise dominated by $Mf$. Hence,
by Lemma~\ref{le:lattice},
\[
\|M^\#f\|_{p(\cdot)}\le {\rm const}\|f\|_{p(\cdot)}
\quad\mbox{for}\quad f\in L^{p(\cdot)}(\R^n)
\]
whenever $p\in\cM(\R^n)$.
The converse is also true. For constant $p$ this fact goes back to
Fefferman and Stein (see e.g. \cite[Chap.~IV, Section~2.2]{S93}).
The variable exponent analogue of the Fefferman-Stein theorem
was proved by Diening and R\r{u}\v zi\v cka \cite{DR03}.
\begin{theorem}{\em (see \cite[Theorem~3.6]{DR03} or \cite[Theorem~6.2.5]{DHHR11})}
\label{th:sharp}
If $p,p'\in\cM(\R^n)$, then there exists a constant $C_\#(p)>0$
such that for all $f\in L^{p(\cdot)}(\R^n)$,
\[
\|f\|_{p(\cdot)}\le C_\#(p)\|M^\#f\|_{p(\cdot)}.
\]
\end{theorem}
\subsection{Duality and left-openness of the class \boldmath{$\cM(\R^n)$}}
Let $1\le q<\infty$. Given $f\in L_{\rm loc}^q(\R^n)$, the {\em $q$-th
maximal operator} is defined by
\[
M_qf(x):=\sup_{Q\ni x}\left(\frac{1}{|Q|}\int_Q|f(y)|^q dy\right)^{1/q},
\]
where the supremum is taken over all cubes $Q\subset\R^n$ containing $x$.
For $q=1$ this is the usual Hardy-Littlewood maximal operator. Diening \cite{D05}
established the following deep duality and left-openness result for the class
$\cM(\R^n)$.
\begin{theorem}{\em (see \cite[Theorem~8.1]{D05} or \cite[Theorem~5.7.2]{DHHR11})}
\label{th:Diening}
Let $p:\R^n\to[1,\infty]$ be a measurable function satisfying \eqref{eq:exponents}.
The following statements are equivalent:
\begin{enumerate}
\item[{\rm(a)}]
$M$ is bounded on $L^{p(\cdot)}(\R^n)$;
\item[{\rm(b)}]
$M$ is bounded on $L^{p'(\cdot)}(\R^n)$;
\item[{\rm(c)}]
there exists an $s\in(1/p_-,1)$ such that $M$ is bounded on $L^{sp(\cdot)}(\R^n)$;
\item[{\rm(d)}]
there exists a  $q\in (1,\infty)$ such that $M_q$ is bounded on $L^{p(\cdot)}(\R^n)$.
\end{enumerate}
\end{theorem}
\subsection{The crucial pointwise estimate}
\label{sec:pointwise} One of the main steps in the proof of
Theorem~\ref{th:boundedness} is the following pointwise estimate.
\begin{theorem}{\em (see \cite[Theorem~3.3]{MRS11})}
\label{th:pointwise}
Let $1<q<\infty$ and $a\in S_{\rho,\delta}^m$ with $0<\rho\le 1$, $0\le\delta<1$,
and $m=n(\rho-1)$. For every $u\in C_0^\infty(\R^n)$,
\[
M^\#(\Op(a)u)(x)\le C(q,a)M_qu(x)\quad (x\in\R^n),
\]
where $C(q,a)$ is a positive constant depending only on  $q$  and
the symbol $a$.
\end{theorem}
This theorem generalizes the pointwise estimate by Miller
\cite[Theorem~2.8]{M82} for $a\in S_{1,0}^0$ and by \'Alvarez and
Hounie \cite[Theorem~4.1]{AH90} for $a\in S_{\rho,\delta}^m$ with
the parameters satisfying $0<\delta\le\rho\le 1/2$ and $m\le
n(\rho-1)$.

Let $0<s<1$. One of the main steps in the Rabinovich and Samko's proof \cite{RS08}
of the boundedness on $L^{p(\cdot)}(\R^n)$ of the operator $\Op(a)$ with
$a\in S_{1,0}^0$ is another pointwise estimate
\[
M^\#(|\Op(a)u|^s)(x)\le C[Mu(x)]^s\quad (x\in\R^n)
\]
for all $u\in C_0^\infty(\R^n)$, where $C$ is a positive constant
independent of $u$. It was proved in \cite[Corollary~3.4]{RS08} following
the ideas of \'Alvarez and P\'erez \cite{AP94}, where the same estimate
is obtained for the Calder\'on-Zygmund singular integral operator in place of
the pseudodifferential operator $\Op(a)$.
\subsection{Proof of Theorem~\ref{th:boundedness}}
\label{sec:proof} Suppose $p\in\cM(\R^n)$. Then, by
Theorem~\ref{th:Diening}, $p'\in\cM(\R^n)$ and there exists a
number $q\in(1,\infty)$ such that $M_q$ is bounded on
$L^{p(\cdot)}(\R^n)$. In other words, there exists a positive
constant $\widetilde{C}(p,q)$ depending only on $p$ and $q$ such
that for all $u\in L^{p(\cdot)}(\R^n)$,
\begin{equation}\label{eq:proof-1}
\|M_qu\|_{p(\cdot)}\le \widetilde{C}(p,q)\|u\|_{p(\cdot)}.
\end{equation}
From Theorem~\ref{th:sharp}
it follows that there exists a constant $C_\#(p)$ such that
for all $u\in C_0^\infty(\R^n)$,
\begin{equation}\label{eq:proof-2}
\|\Op(a)u\|_{p(\cdot)}\le C_\#(p)\|M^\#(\Op(a)u)\|_{p(\cdot)}.
\end{equation}
On the other hand, from Theorem~\ref{th:pointwise} and
Lemma~\ref{le:lattice} we obtain that there exists a positive constant
$C(q,a)$, depending only on $q$ and $a$, such that
\begin{equation}\label{eq:proof-3}
\|M^\#(\Op(a)u)\|_{p(\cdot)}\le C(q,a)\|M_qu\|_{p(\cdot)}.
\end{equation}
Combining \eqref{eq:proof-1}--\eqref{eq:proof-3}, we arrive at
\[
\|\Op(a)u\|_{p(\cdot)}\le C_\#(p) C(q,a) \widetilde{C}(p,q)\|u\|_{p(\cdot)}
\]
for all $u\in C_0^\infty(\R^n)$. It remains to recall that $C_0^\infty(\R^n)$
is dense in $L^{p(\cdot)}(\R^n)$ (see Lemma~\ref{le:density}).
\qed

\section{Fredholmness of the operator \boldmath{${\rm Op}(a)$}}
\label{sec:Fredholmness}
\subsection{Interpolation theorem}\label{sec:interpolation}
For a Banach space $X$, let $\cB(X)$ and $\cK(X)$ denote the Banach
algebra of all bounded linear operators and its ideal of all compact
operators on $X$, respectively.
\begin{theorem}\label{th:interpolation}
Let $p_j:\R^n\to[1,\infty]$, $j=0,1$, be a.e. finite measurable
functions, and let $p_\theta:\R^n\to[1,\infty]$ be defined for
$\theta\in[0,1]$ by
\[
\frac{1}{p_\theta(x)}=\frac{\theta}{p_0(x)}+\frac{1-\theta}{p_1(x)}\quad
(x\in\R^n).
\]
Suppose $A$ is a linear operator defined on $L^{p_0}(\R^n)\cup
L^{p_1}(\R^n)$.
\begin{enumerate}
\item[{\rm(a)}] If $A\in\cB(L^{p_j}(\R^n))$ for $j=0,1$, then
    $A\in\cB(L^{p_\theta(\cdot)}(\R^n))$ for all $\theta\in[0,1]$ and
\[
\|A\|_{\cB(L^{p_\theta(\cdot)}(\R^n))} \le 4
\|A\|_{\cB(L^{p_0(\cdot)}(\R^n))}^\theta
\|A\|_{\cB(L^{p_1(\cdot)}(\R^n))}^{1-\theta}.
\]

\item[{\rm(b)}]   If
    $A\in\cK(L^{p_0(\cdot)}(\R^n))$ and
    $A\in\cB(L^{p_1(\cdot)}(\R^n))$, then
    $A\in\cK(L^{p_\theta(\cdot)}(\R^n))$ for all $\theta\in(0,1)$.
\end{enumerate}
\end{theorem}
Part (a) is proved in \cite[Corollary~7.1.4]{DHHR11} under the
more general assumption that $p_j$ may take infinite values on
sets of positive measure (and in the setting of arbitrary measure
spaces). Part (b) was proved in \cite[Proposition~2.2]{RS08} under
the additional assumptions that $p_j$ satisfy
\eqref{eq:exponents}--\eqref{eq:Hoelder-infinity}. It follows
without these assumptions from a general interpolation theorem by
Cobos, K\"uhn, and Schonbeck \cite[Theorem~3.2]{CKS92} for the
complex interpolation method for Banach lattices satisfying the
Fatou property. Indeed, the complex interpolation space
$[L^{p_0(\cdot)}(\R^n),L^{p_1(\cdot)}(\R^n)]_{1-\theta}$ is
isomorphic to the variable Lebesgue space
$L^{p_\theta(\cdot)}(\R^n)$ (see \cite[Theorem~7.1.2]{DHHR11}),
and $L^{p_j(\cdot)}(\R^n)$ have the Fatou property (see
\cite[p.~77]{DHHR11}).
\subsection{Calculus of pseudodifferential operators}\label{sec:calculus}
Let $m\in\Z$ and $OPSO^m$ be the class of all pseudodifferential
operators $\Op(a)$ with $a\in SO^m$. By analogy with
\cite[Section~2]{G70} one can get the following {\em composition
formula}  (see also \cite[Theorem~6.2.1]{R98} and
\cite[Chap.~4]{RRS04}).
\begin{proposition}\label{pr:composition}
If $\Op(a_1)\in OPSO^{m_1}$ and $\Op(a_2)\in OPSO^{m_2}$, then
their product $\Op(a_1)\Op(a_2)=\Op(\sigma)$ belongs to
$OPSO^{m_1+m_2}$ and its symbol $\sigma$ is given by
\[
\sigma(x,\xi)=a_1(x,\xi)a_2(x,\xi)+c(x,\xi), \quad x,\xi\in\R^n,
\]
where $c\in SO_0^{m_1+m_2-1}$.
\end{proposition}
\begin{proposition}\label{pr:compactness}
Let $1<q<\infty$. If $c\in SO_0^{-1}$, then
$\Op(c)\in\cK(L^q(\R^n))$.
\end{proposition}
\begin{proof}
From Theorem~\ref{th:boundedness} it follows that
$\Op(c)\in\cB(L^q(\R^n))$ for all constant exponents
$q\in(1,\infty)$. By \cite[Theorem~3.2]{G70},
$\Op(c)\in\cK(L^2(\R^n))$. Hence, by the Krasnoselskii
interpolation theorem (Theorem~\ref{th:interpolation}(b) for
constant $p_j$ with $j=0,1$), $\Op(c)\in\cK(L^q(\R^n))$ for all
$q\in(1,\infty)$.
\end{proof}
\subsection{Proof of Theorem~\ref{th:sufficiency}}\label{sec:sufficiency}
The idea of the proof is borrowed from \cite[Theorem~3.4]{G70} and
\cite[Theorem~6.1]{RS08}. Let $\varphi\in
C_0^\infty(\R^n\times\R^n)$ be such that $\varphi(x,\xi)=1$ if
$|x|+|\xi|\le 1$ and $\varphi(x,\xi)=0$ if $|x|+|\xi|\ge 2$. For
$R>0$, put
\[
\varphi_R(x,\xi)=\varphi(x/R,\xi/R),\quad x,\xi\in\R^n.
\]
From \eqref{eq:sufficiency-1} it follows that there exists an
$R>0$ such that
\[
\inf_{|x|+|\xi|\ge R}|a(x,\xi)|>0.
\]
Then it is not difficult to check that
\[
b_R(x,\xi):=\left\{\begin{array}{lll}
\displaystyle\frac{1-\varphi_R(x,\xi)}{a(x,\xi)} &\mbox{if}& |x|+|\xi|\ge R,\\
0&\mbox{if}& |x|+|\xi|<R,
\end{array}\right.
\]
belongs to $SO^0$. It is also clear that $\varphi_R\in SO^0$.

From Proposition~\ref{pr:composition} it follows that there exists
a function $c\in SO_0^{-1}$ such that
\begin{equation}\label{eq:sufficiency-2}
\Op(ab_R)-\Op(a)\Op(b_R)=\Op(c).
\end{equation}
On the other hand, since
\[
a(x,\xi)b_R(x,\xi)=1-\varphi_R(x,\xi),\quad x,\xi\in\R^n,
\]
we have
\begin{equation}\label{eq:sufficiency-3}
\Op(ab_R)=\Op(1-\varphi_R)=I-\Op(\varphi_R).
\end{equation}
Combining \eqref{eq:sufficiency-2}--\eqref{eq:sufficiency-3}, we
get
\begin{equation}\label{eq:sufficiency-4}
I-\Op(a)\Op(b_R)=\Op(\varphi_R)+\Op(c)=\Op(\varphi_R+c).
\end{equation}

Since $p\in\cM^*(\R^n)$, there exist $p_0\in(1,\infty)$,
$\theta\in(0,1)$, and $p_1\in\cM(\R^n)$ such that
\[
\frac{1}{p(x)}=\frac{\theta}{p_0}+\frac{1-\theta}{p_1(x)}
\quad(x\in\R^n).
\]
From Theorem~\ref{th:boundedness} we conclude that all
pseudodifferential operators considered above are bounded on
$L^{p_0}(\R^n)$, $L^{p(\cdot)}(\R^n)$, and $L^{p_1(\cdot)}(\R^n)$.
Since $\varphi_R+c\in SO_0^{-1}$, from
Proposition~\ref{pr:compactness} it follows that
$\Op(\varphi_R+c)\in\cK(L^{p_0}(\R^n))$. Then, by
Theorem~\ref{th:interpolation}(b),
$\Op(\varphi_R+c)\in\cK(L^{p(\cdot)}(\R^n))$. Therefore, from
\eqref{eq:sufficiency-4} it follows that $\Op(b_R)$ is a right
regularizer for $\Op(a)$. Analogously it can be shown that
$\Op(b_R)$ is also a left regularizer for $\Op(a)$. Thus $\Op(a)$
is Fredholm on $L^{p(\cdot)}(\R^n)$. \qed



\begin{thebibliography}{99}
\bibitem{AH90}
J. \'Alvarez and J. Hounie,
\textit{Estimates for the kernel and continuity properties of pseudo-differential operators.}
\href{http://dx.doi.org/10.1007/BF02387364}
{Ark. Mat. \textbf{28} (1990), 1–-22}.

\bibitem{AP94}
J.~\'Alvarez and C.~P\'erez,
\textit{Estimates with $A_\infty$ weights for various singular integral operators.}
Boll. Un. Mat. Ital. A (7) \textbf{8} (1994), 123–-133.

\bibitem{CCF07}
C.~Capone, D.~Cruz-Uribe, and A.~Fiorenza,
\textit{The fractional maximal operator and fractional integrals on variable $L^p$ spaces.}
\href{http://projecteuclid.org/euclid.rmi/1204128298}
{Rev. Mat. Iberoamericana \textbf{23} (2007), 743--770}.

\bibitem{CKS92}
F.~Cobos, T.~K\"uhn, and T.~Schonbek,
\textit{One-sided compactness results for Aronszajn-Gagliardo functors.}
\href{http://dx.doi.org/10.1016/0022-1236(92)90049-O}
{J. Funct. Analysis \textbf{106} (1992), 274--313}.

\bibitem{CFN03}
D.~Cruz-Uribe, A.~Fiorenza, and C.~J.~Neugebauer,
\textit{The maximal function on variable $L^{p}$ spaces.}
Ann. Acad. Sci. Fenn. Math. \textbf{28} (2003), 223-–238.

\bibitem{CFN04}
D.~Cruz-Uribe, A.~Fiorenza, and C.~J.~Neugebauer,
\textit{Corrections to: ``The maximal function on variable $L^{p}$ spaces".}
Ann. Acad. Sci. Fenn. Math. \textbf{29} (2004), 247-–249.

\bibitem{D04}
L.~Diening,
\textit{Maximal function on generalized Lebesgue spaces $L^{p(\cdot)}$.}
Math. Inequal. Appl. \textbf{7} (2004), 245-–253.

\bibitem{D05}
L.~Diening,
\textit{Maximal function on Musielak-Orlicz spaces and generlaized Lebesgue spaces.}
\href{http://dx.doi.org/10.1016/j.bulsci.2003.10.003}
{Bull. Sci. Math.  \textbf{129}  (2005),  657--700}.

\bibitem{DHHR11}
L.~Diening, P.~Harjulehto, P.~H\"ast\"o, and M.~R\r{u}\v zi\v cka,
\textit{Lebesgue and Sobolev Spaces with Variable Exponents}.
\href{http://dx.doi.org/10.1007/978-3-642-18363-8}
{Lecture Notes in Mathematics \textbf{2017}, 2011}.

\bibitem{DR03}
L.~Dieinig and M.~R\r{u}\v zi\v cka,
\textit{Calder\'on-Zygmund operators on generalized Lebesgue spaces $L^{p(\cdot)}$
and problems related to fluid dynamics}.
\href{http://dx.doi.org/10.1515/crll.2003.081}
{J. Reine Angew. Math., \textbf{563} (2003), 197--220}.

\bibitem{G70}
V.~V.~Grushin,
\textit{Pseudodifferential operators on $\R^n$ with bounded symbols}.
\href{http://dx.doi.org/10.1007/BF01075240}
{Funct. Anal. Appl. \textbf{4} (1970), 202--212}.

\bibitem{H67}
L.~H\"ormander,
\textit{Pseudo-differential operators and hypoelliptic equations.}
In: ``Singular integrals (Proc. Sympos. Pure Math., Vol. X, Chicago, Ill., 1966)",
pp. 138–-183. Amer. Math. Soc., Providence, R.I., 1967.

\bibitem{KS11-example}
A.~Yu.~Karlovich and I.~M.~Spitkovsky,
\textit{On an interesting class of variable exponents},
submitted.

\bibitem{KR91}
O.~Kov\'a{\v c}ik and J.~R\'akosn\'ik,
\textit{On spaces $L\sp {p(x)}$ and $W\sp {k,p(x)}$}.
\href{http://hdl.handle.net/10338.dmlcz/102493}
{Czechoslovak Math. J. \textbf{41(116)} (1991), no. 4, 592--618}.

\bibitem{L05}
A.~K.~Lerner, \textit{Some remarks on the Hardy-Littlewood maximal
function on variable $L^p$ spaces}.
\href{http://dx.doi.org/10.1007/s00209-005-0818-5}
{Math. Z. \textbf{251} (2005), no. 3, 509–-521}.

\bibitem{MRS11}
N.~Michalowski, D.~Rule, and W.~Staubach,
\textit{Weighted $L^p$ boundedness of pseudodifferential operators and applications.}
\href{http://dx.doi.org/10.4153/CMB-2011-122-7}
{Canad. Math. Bull. (2011), doi:10.4153/CMB-2011-122-7}.

\bibitem{M82}
N.~Miller,
\textit{Weighted Sobolev spaces and pseudodifferential
operators with smooth symbols.}
\href{http://dx.doi.org/10.1090/S0002-9947-1982-0637030-4}
{Trans. Amer. Math. Soc. \textbf{269} (1982), 91--109}.

\bibitem{N04}
A.~Nekvinda,
\textit{Hardy-Littlewood maximal operator on $L^{p(x)}(\R^n)$.}
Math. Inequal. Appl. \textbf{7} (2004), 255--265.

\bibitem{N08}
A.~Nekvinda,
\textit{Maximal operator on variable Lebesgue spaces for almost monotone radial exponent}.
\href{http://dx.doi.org/10.1016/j.jmaa.2007.04.047}
{J. Math. Anal. Appl. \textbf{337} (2008), 1345-–1365}.

\bibitem{PR01}
L.~Pick and M.~R\r{u}\v zi\v cka,
\textit{An example of a space of Lp(x) on which the Hardy-Littlewood maximal operator is not bounded}.
\href{http://dx.doi.org/10.1016/S0723-0869(01)80023-2}
{Expo. Math. \textbf{19} (2001), 369--371}.

\bibitem{R98}
V.~S.~Rabinovich,
\textit{An itroductory course on pseudodifferential operators}.
Textos de Matem\'atica, Instituto Superior T\'ecnico, Lisboa, 1998.

\bibitem{RRS04}
V.~S.~Rabinovich, S.~Roch, and B.~Silbermann,
\textit{Limit Operators and Their Applications in Operator Theory}.
Operator Theory: Advances and Applications, vol. 150. Birkh\"auser, Basel, 2004.

\bibitem{RS08}
V.~S.~Rabinovich and S.~G.~Samko,
\textit{Boundedness and Fredholmness of pseudodifferential operators in variable exponent spaces.}
\href{http://dx.doi.org/10.1007/s00020-008-1566-9}
{Integral Equations Operator Theory \textbf{60} (2008), 507--537}.

\bibitem{RS11}
V.~S.~Rabinovich and S.~G.~Samko,
\textit{Pseudodifferential operators approach to singular integral operators in weighted
variable exponent Lebesgue spaces on Carleson curves.}
\href{http://dx.doi.org/10.1007/s00020-010-1848-x} {Integral
Equations Operator Theory \textbf{69} (2011), 405--444}.

\bibitem{S93}
E.~Stein,
\textit{Harmonic Analysis: Real-Variable Methods, Orthogonality, and Oscillatory Integrals.}
Princeton Uinversity Press, Princeton, NJ, 1993.
\end{thebibliography}
\end{document}